\pdfoutput=1
\RequirePackage{ifpdf}
\ifpdf % We~are running pdfTeX in pdf mode
\documentclass[pdftex]{sigma}
\else
\documentclass{sigma}
\fi

\usepackage{tikz-cd}
\usetikzlibrary{matrix}

\newcommand\reallywidehat[1]{%
	\savestack{\tmpbox}{\stretchto{%
			\scaleto{%
				\scalerel*[\widthof{\ensuremath{#1}}]{\kern-.6pt\bigwedge\kern-.6pt}%
				{\rule[-\textheight/2]{1ex}{\textheight}}%
			}{\textheight}%
		}{0.5ex}}%
	\stackon[1pt]{#1}{\tmpbox}%
}

\newcommand{\C}{{\mathbb C}}
\newcommand{\R}{{\mathbb R}}

\renewcommand{\H}{{\mathbb H}}
\newcommand{\6}{\partial}

\newcommand{\arrow}{{\:\longrightarrow\:}}

\newcommand{\Aut}{\operatorname{Aut}}

\renewcommand{\Re}{\operatorname{Re}}
\renewcommand{\Im}{\operatorname{Im}}

\makeatletter
\newcommand*{\mfaktor}[3][]{%
 { \mathpalette{\mfaktor@impl@}{{#1}{#2}{#3}} }%
}
\newcommand*{\mfaktor@impl@}[2]{\mfaktor@impl#1#2}
\newcommand*{\mfaktor@impl}[4]{%
 \settoheight{\faktor@zaehlerhoehe}{\ensuremath{#1#2{#3}}}%
 \settoheight{\faktor@nennerhoehe}{\ensuremath{#1#2{#4}}}%
 \raisebox{-0.5\faktor@zaehlerhoehe}{\ensuremath{#1#2{#3}}}%
 \mkern-4mu\diagdown\mkern-5mu%
 \raisebox{0.5\faktor@nennerhoehe}{\ensuremath{#1#2{#4}}}%
}
\makeatother

\numberwithin{equation}{section}
\newtheorem{Theorem}{Theorem}[section]
\newtheorem*{Theorem*}{Theorem}

\newtheorem{Proposition}[Theorem]{Proposition}
\newtheorem{Conjecture}[Theorem]{Conjecture}

{ \theoremstyle{definition}
\newtheorem{Definition}[Theorem]{Definition}

\newtheorem{Example}[Theorem]{Example}
\newtheorem{Remark}[Theorem]{Remark}
}

\begin{document}
%\allowdisplaybreaks

\newcommand{\arXivNumber}{2502.19520}

\renewcommand{\PaperNumber}{069}

\FirstPageHeading

\ShortArticleName{Curves on Endo--Pajitnov Manifolds}
\ArticleName{Curves on Endo--Pajitnov Manifolds}

\Author{Cristian CIULIC\u{A}}
\AuthorNameForHeading{C.~Ciulic\u{a}}

\Address{Faculty of Mathematics and Computer Science, University of Bucharest,\\ 14 Academiei Str., 010014 Bucharest, Romania}
\Email{\href{mailto:cristiciulica@yahoo.com}{cristiciulica@yahoo.com}}

\ArticleDates{Received February 28, 2025, in final form August 06, 2025; Published online August 13, 2025}

\Abstract{Endo--Pajitnov manifolds are generalizations to higher dimensions of the Inoue surfaces $S^M$. We study the existence of complex submanifolds in Endo--Pajitnov manifolds. We identify a class of these manifolds that do contain compact complex submanifolds and establish an algebraic condition under which an Endo--Pajitnov manifold contains no compact complex curves.}

\Classification{53C55}

\section{Introduction}
	
Among the surfaces in Kodaira's class VII, the three types of Inoue surfaces~\cite{inoue}, play a prominent role. They are compact non-K\"ahler surfaces with no non-trivial meromorphic functions and without complex curves.

 The Inoue surfaces of type $S^M$ are solvmanifolds, quotients of $\H\times\C$, where $\H$ is the Poincar\'e half-plane, by a group constructed out of a matrix $M\in \textrm{SL}(3, \mathbb{Z})$ with one real (irrational) eigenvalue $\alpha>1$ and two complex conjugate ones, $\beta$, $\bar\beta$. Denoting with $(a_i)$, respectively $(b_i)$, a real eigenvector of $\alpha$, respectively an eigenvector of $\beta$, we can define $g_0(w,z)=(\alpha w, \beta z)$ and~${g_i(w,z)=(w+a_i, z+b_i)}$, $i=1,2,3$, and let $G_M$ be the group generated by $g_0$, $g_1$, $g_2$, $g_3$. Then the Inoue surface $S^M$ is $G_M\backslash\H\times\C$.

 In 2005, the surfaces $S^M$ were generalized to higher dimensions by Oeljeklaus and Toma~\cite{ot}. Each such manifold is covered by $\H^s\times \C^t$ and it is associated to a number field with $s$ real places and $t$ complex ones. The Oeljeklaus--Toma (OT) manifolds are non-K\"ahler and contain no compact complex curves~\cite{v_1}, and no compact complex submanifolds of dimension 2 except Inoue surfaces \cite{v_2}. Moreover, the OT manifolds which admit locally conformally K\"ahler metrics do not have non-trivial meromorphic functions and hence they do not admit compact complex submanifolds~\cite{ov}.

 In 2019, Endo and Pajitnov~\cite{pajitnov1} proposed another generalization of the Inoue surfaces $S^M$ to higher dimensions, again based on an integer matrix $M$ with special requirements on its eigenvalues, just like the original construction. They proved that these new manifolds are non-K\"ahler and, if $M$ is diagonalizable, then some of these manifolds are biholomorphic to OT manifolds. Further topological and metric properties of the Endo--Pajitnov manifolds were discussed in~\cite{cos}.

 In this note, we study the existence of complex submanifolds in Endo--Pajitnov manifolds. We~describe a class of Endo--Pajitnov manifolds which contain complex submanifolds, specifically complex tori (Theorem~\ref{ex_th}). On the other hand, we determine an algebraic condition which prohibits the existence of compact complex curves (Theorem~\ref{th_main}). We also obtain a result (Proposition~\ref{sup}) regarding the existence of surfaces in an Endo--Pajitnov manifold from the~class of those without complex curves.

\section{Endo--Pajitnov manifolds}
\label{EP_def_sec}

In this section, we recall the construction of the Endo--Pajitnov manifolds, as introduced in~\cite{pajitnov1}.

Let $n > 1$ and $M \in \textrm{SL}(2n+1, \mathbb{Z})$ such that the eigenvalues of $M$ are $\alpha, \beta_1, \dots, \beta_k$, $\overline{\beta}_1, \dots, \overline{\beta}_k$ with $\alpha>0$, $\alpha \neq 1$ and $\Im(\beta_j)>0$.

Denote by $V$ the eigenspace corresponding to $\alpha$ and set
\begin{equation*}
	\begin{split}
		&W(\beta_j)=\bigl\{x \in \mathbb{C}^{2n+1}\mid \exists \ N \in \mathbb{N} \text{ such that } (M-\beta_j I)^Nx=0\bigr\},\\
	 &W=\bigoplus \limits_{j=1}^{k} W(\beta_j),\qquad \overline{W}=\bigoplus\limits_{j=1}^{k} W\big(\overline{\beta}_j\big).
	\end{split}
\end{equation*}
We then have $\mathbb{C}^{2n+1}=V \oplus W \oplus \overline{W}$.
Let $a \in \R^{2n+1}$ be a non-zero eigenvector corresponding to~$\alpha$ and fix a~basis $\{b_1, \dots, b_n\}$ in~$W$,
\[
a=\big(a^1, a^2, \dots, a^{2n+1}\big)^{\mathsf{T}},\qquad \ b_i=\big( b_{i}^1, b_{i}^2, \dots, b_{i}^{2n+1}\big)^{\mathsf{T}}, \qquad 1 \leq i\leq n.
\]
For any $1 \le i \le 2n + 1$, we let $u_i = \big(a^i, b_1^i, \dots , b_n^i\big) \in \R \times \C^n \simeq \R^{2n+1}$.
Note that $\{u_1, \dots , u_{2n+1} \}$ are linearly independent over $\R$, since $\bigl\{a, b_1, \dots, b_n, \overline{b}_1, \dots, \overline{b}_n\bigr\}$ is a~basis of~$\mathbb{C}^{2n+1}$.

Let now $f_M\colon W \longrightarrow W$ be the restriction of the multiplication by $M$ on $W$ and $R$ the matrix of $f_M$ with respect to the basis $\{b_1, \dots , b_n\}$. Let $\mathbb{H}$ be the Poincar\'e upper half-plane, and consider the automorphisms $g_0, g_1, \dots, g_{2n+1}\colon \mathbb{H} \times \mathbb{C}^n \longrightarrow \mathbb{H} \times \mathbb{C}^n$,
\begin{gather*}
 g_0(w,z)=\big(\alpha w, R^{\mathsf{T}}z\big), \qquad g_i(w,z)=(w,z)+u_i,\qquad w \in \mathbb{H}, \quad z \in \mathbb{C}^n,\quad 1 \leq i \leq 2n+1.
\end{gather*}
These automorphisms are well defined because $\alpha > 0$ and the first component of $u_i$ is $a^i \in \R$.

Let $G_M$ be the subgroup of Aut$(\mathbb{H} \times \mathbb{C}^n)$ generated by $g_0, g_1, \dots, g_{2n+1}$.

\begin{Theorem}[{\cite{pajitnov1}}] The action of $G_M$ on $\mathbb{H} \times \mathbb{C}^n$ is free and properly discontinuous. Hence, the quotient $T_M:= (\mathbb{H} \times \mathbb{C}^n)/G_M$ is a compact complex manifold of complex dimension $n+1$, with $\pi_1(T_M) \simeq G_M$.
\end{Theorem}

\begin{Definition}
The above quotient $T_M:= (\mathbb{H} \times \mathbb{C}^n)/G_M$ is called an \emph{Endo--Pajitnov manifold}.
\end{Definition}

\begin{Remark}
In the same paper, the authors prove that
\begin{itemize}\itemsep=0pt
\item If $M$ is diagonalizable, then some $T_M$ are biholomorphic to OT manifolds \cite[Proposition~5.3]{pajitnov1}.
\item If $M$ is not diagonalizable, then $T_M$ cannot be biholomorphic to any OT manifold \cite[Proposition~5.6]{pajitnov1}.
\end{itemize}
\end{Remark}

\section{A class of Endo--Pajitnov manifolds containing submanifolds}

We shall identify a class of Endo--Pajitnov manifolds that admit compact complex submanifolds. The idea is to define a holomorphic submersion from $T_M$ to another complex manifold; the fibers will be the complex submanifolds we look for. More precisely, these submanifolds will be complex~tori.
The existence of such a structure depends on a suitable choice of the initial matrix~$M$.

Let $n > 1$, $k \geq 1$ and $M \in \textrm{SL}(2n+1, \mathbb{Z})$ be a matrix which can be written in block form:
\[
M = \begin{pmatrix}
N & 0 \\
0 & P
\end{pmatrix}
\]
where \( P \) is a square matrix of dimension \( 2k \), \( N \in \textrm{SL}(2(n-k)+1, \mathbb{Z}) \), such that
\[
{\rm Spec}(N) = \bigl\{ \alpha, \beta_1, \dots, \beta_N, \overline{\beta}_1, \dots, \overline{\beta}_N \mid \alpha \in \mathbb{R},\, \alpha > 0,\, \alpha \neq 1,\, \Im(\beta_j) > 0 \bigr\}
\]
and
\[
{\rm Spec}(P) = \bigl\{ \beta_{N+1}, \dots, \beta_{P}, \overline{\beta}_{N+1}, \dots, \overline{\beta}_{P} \mid \Im(\beta_j) > 0 \bigr\}.
\]
It is clear that $M$ satisfies the conditions required in the construction of Endo--Pajitnov manifolds.
In this case, the matrix $R$ (see Section \ref{EP_def_sec}) is a block diagonal matrix.

Denote by $W^M(\beta_j)$, $W^N(\beta_j)$, and $W^P(\beta_j)$ the generalised eigenspaces of $\beta_j$ for $M$, $N$, and $P$ respectively. We pick a basis $\{ b_1, \dots, b_{2n+1} \}$ in $W$ which comes from bases in each $W^M(\beta_j)$ that in turn come from bases in $W^N(\beta_j)$ and $W^P(\beta_j)$, using the fact that $W^M(\beta_j) = W^N(\beta_j) \oplus W^P(\beta_j)$.

The diffeomorphisms $g_0, g_1, \dots, g_{2n+1}\colon \mathbb{H} \times \mathbb{C}^n \longrightarrow \mathbb{H} \times \mathbb{C}^n$ can be written explicitly as follows:
\begin{gather}
		g_0(w, (z_1, \dots, z_n)) = \bigl(\alpha w, R^{\mathsf{T}}z\bigr),\label{aut_ex}\\
		g_i(w, z) = (w, z) +
\begin{cases}
\big(a^i, b_1^i, \dots, b_{n-k}^i, 0, \dots, 0\big), & 1 \leq i \leq 2(n-k)+1, \\
\big(0, 0, \dots, 0, b_{n-k+1}^i, \dots, b_n^i\big), & i > 2(n-k)+1,
\end{cases}
\quad w \in \mathbb{H}, z \in \mathbb{C}^n.\nonumber
\end{gather}

Since $R$ is a block diagonal matrix and $g_i$ acts independently on each component for any $ 1 \leq i \leq 2n+1$, and owing to the special form of the automorphisms \eqref{aut_ex}, there exist ${\tilde{g_i}\colon \mathbb{H} \times \mathbb{C}^{n-k} \longrightarrow \mathbb{H} \times \mathbb{C}^{n-k}}$ such that the following diagram commutes:
\[
\begin{tikzcd}
	\mathbb{H} \times \mathbb{C}^{n} \arrow[r, "g_i"] \arrow[d, "{\rm pr}_{n-k}"'] & \mathbb{H} \times \mathbb{C}^{n} \arrow[d, "{\rm pr}_{n-k}"] \\
	\mathbb{H} \times \mathbb{C}^{n-k} \arrow[r, "\tilde{g_i}"] & \mathbb{H} \times \mathbb{C}^{n-k}.
\end{tikzcd}
\]

Let $\Gamma_N$ the subgroup of ${\rm Aut}\bigl(\mathbb{H} \times \mathbb{C}^{n-k}\bigr)$ generated by $\tilde g_0, \dots, \tilde g_{2(n-k)+1}$.
Since the matrix $N$ satisfies the conditions in the construction in Section \ref{EP_def_sec}, the action of $\Gamma_N$ generates an Endo--Pajitnov manifold $T_N$ of dimension $n-k+1$, $T_N=\big(\mathbb{H} \times \mathbb{C}^{n-k}\big)/ {\Gamma_N}$.

It is known that $T_M$ has a solvmanifold structure, $T_M \simeq G/{\Gamma}$ \cite[Theorem~3.1]{cos}. Let us denote by $\mathfrak{g}$ the Lie algebra of $G$.
Our goal is to show that $T_M$ is the total space of a holomorphic fiber bundle with base $T_N$. To achieve this, we use the fact that both $T_M$ and $T_N$ are solvmanifolds and we analyze the structure of $\mathfrak{g}$. The idea is to write $\mathfrak{g}$ as a semidirect product between an abelian ideal that will correspond to the fiber, and a complementary subalgebra corresponding to the base $T_N$.

This structure naturally extends to the corresponding Lie group level.
We also describe the lattice $\Gamma$ as a semidirect product between a lattice generating $T_N$ and a lattice consisting only of translations. The compatibility between the semidirect product structures of $G$ and $\Gamma$ allows us to define a holomorphic submersion $\pi$ between $T_M$ and $T_N$, and the fiber will be a complex torus.\looseness=-1

Consider the subspace
\[
\mathfrak{h} = \bigl\langle Y_{2n-k + 1} + \mathrm{i} Y_{n-k+1}, \dots , Y_{2n} + \mathrm{i} Y_n, \overline{Y_{2n-k + 1} + \mathrm{i} Y_{n-k+1}}, \dots , \overline{Y_{2n} + \mathrm{i} Y_n} \bigr\rangle_\C.
 \]
A direct computation shows that $\mathfrak{h}$ is an abelian ideal of $\mathfrak{g}$. Indeed, from the structure equations of $\mathfrak{g}$, it is sufficient to prove that $\bigl[A,Y_{2n-k + j} + \mathrm{i} Y_{n-k+j}\bigr] \in \mathfrak{h}$, for any $ 1 \leq j \leq k$.

Let us prove the case $j=1$. We have{\samepage
\begin{align*}
[A, Y_{2n-k+1} + \mathrm{i} Y_{n-k+1}]
 &= -\sum_{i \leq n-k} \Im \Delta_{i, n-k+1} \cdot Y_i
 + \sum_{i \leq n-k} \Re \Delta_{i, n-k+1} \cdot Y_{n+i} \\
 &\quad + \mathrm{i} \biggl( \sum_{i \leq n-k} \Re \Delta_{i, n-k+1} \cdot Y_i
 + \sum_{i \leq n-k} \Im \Delta_{i, n-k+1} \cdot Y_{n+i} \biggr),
\end{align*}
where $\Delta = \log R^{\mathsf{T}}$ (see~\cite[Theorem~3.1]{cos}).}

Since $R$ is a block diagonal matrix, $\Delta$ inherits this structure and, thus $\Delta_{i, n-k+1}=0$, for all $i\leq n-k$.
Therefore,
\[
[A, Y_{2n-k+1} + \mathrm{i} Y_{n-k+1}]=\Delta_{n-k+1, n-k+1}(Y_{2n-k+1}+\mathrm{i}Y_{n-k+1})\in \mathfrak{h}.
\]
Analogously, we can prove the same thing for any $j$, hence, $\mathfrak{h}$ is ideal of $\mathfrak{g}$ and from the structure equations it is clear that $\mathfrak{h}$ is abelian.

Consider $\mathfrak{l}:=\mathfrak{g}/\mathfrak{h}$. Then $\mathfrak{l}$ is the Lie algebra corresponding to the Endo--Pajitnov manifold~$T_N$.
Moreover,
\[
 \mathfrak{l} = \bigl\langle A+\mathrm{i}X, Y_{n+ 1} + \mathrm{i} Y_{1}, \dots , Y_{2n} + \mathrm{i} Y_n, \overline{A+\mathrm{i}X}, \dots , \overline{ Y_{2n} + \mathrm{i} Y_n} \bigr\rangle_\C
 \]
 is a subalgebra of $\mathfrak{g}$, and $\mathfrak{l}$ acts on $\mathfrak{h}$ via the adjoint representation, giving $\mathfrak{g}=\mathfrak{l} \rtimes_\varphi \mathfrak{h}$.

Let $H:=\C^k$ be the simply-connected complex Lie group corresponding to $\mathfrak{h}$ and $L:=\H \times \C^{n-k}$ the simply-connected complex Lie group corresponding to $\mathfrak{l}$. As differentiable manifolds, $G= L \times H$. Since $H$ is simply-connected, we can identify $\Aut(\mathfrak{h})=\Aut(H)$ and by~\cite[Theorem~3]{bry}, the following diagram commutes:
\[
\begin{tikzcd}
	\mathfrak{l} \arrow[r, "\varphi"] \arrow[d, "{\rm exp}"'] & {\rm Der}(\mathfrak{h}) \arrow[d, "{\rm exp}"] \\
	L \arrow[r, "\tilde{\varphi}"] & \Aut(H),
\end{tikzcd}
\] where $\tilde{\varphi}$ is given by conjugation.
Thus, $L$ acts on $H$ and by Lie's third theorem we have that $G=L \rtimes_{\tilde{\varphi}} H$ with the composition group law
\[
(l_1, h_1) \cdot(l_2, h_2)=(l_1l_2, \tilde{\varphi}(l_1)(h_2)).
\]

Let $\Gamma_P:=\bigl\langle g_{2(n-k)+2}, \dots, g_{2n+1}\bigr\rangle$. Then $\Gamma_P$ is a lattice of $\C^k$.
We will prove that $\Gamma=\Gamma_N \rtimes \Gamma_P$.
We define an action $\rho\colon \Gamma_N \longrightarrow \Aut(\Gamma_P)$ by $\rho(g)(h)=ghg^{-1}$, for any $g \in \Gamma_N$ and $h \in \Gamma_P$. This map is well defined. To verify this, it suffices to show that for every $g_i$ with $0 \leq i \leq 2(n-k)+1$ and every $g_j$ with $2(n-k)+2 \leq j \leq 2n+1$, the conjugate $g_ig_jg_i^{-1} \in \Gamma_P$.

For $i>0$, it is trivial.
For $i=0$, using~\cite[Lemma 2.3]{pajitnov1}, we have
\[
g_0g_jg_0^{-1}=g_1^{m_{j,1}} \cdots g_{2n+1}^{m_{j, 2n+1}}.
\]

Given the block form of $M$, it follows that $m_{j,l}=0$, for any $l \leq 2(n-k)+1$, and thus
\[
g_0g_jg_0^{-1}=g_{2(n-k)+2}^{m_{j, 2(n-k)+2}} \cdots g_{2n+1}m^{j, 2n+1} \qquad \text{for all}\ 2(n-k)+2 \leq j \leq 2n+1.
\]
Hence, we obtain the semidirect product structure, $\Gamma=\Gamma_N \rtimes \Gamma_P$.

Let us consider $p\colon G \longrightarrow L$ the projection $p(h,l)=l$.
Since $p(\Gamma)=\Gamma_N$, this descends to a well-defined map on the quotients:
\begin{align*}
\pi\colon \ G/\Gamma \simeq T_M \longrightarrow L/{\Gamma_N} \simeq T_N,\qquad
\pi(g\Gamma)=p(g)\Gamma_N.
\end{align*}
To see that $\pi$ is well-defined, suppose $g\Gamma=g'\Gamma$, which implies $g^{-1}g' \in \Gamma$. For $g=(l,h)$, $g'=(l',h')$, we have
\begin{equation*}
 g^{-1}g'=\big(l^{-1},l^{-1}h^{-1}l\big) \cdot (l',h')=\big(l^{-1}l', l^{-1}h^{-1}l \cdot \tilde{\varphi}\big(l^{-1}\big)(h')\big) \in \Gamma.
\end{equation*}
Hence, $l^{-1}l' \in \Gamma_N$. So $p(g)^{-1}p(g') \in \Gamma_N$, and therefore $p(g)\Gamma_N=p(g')\Gamma_N$.

Since $\pi$ is induced by the projection $p$ on the first coordinates, it is clearly a holomorphic submersion. Also, it is a proper map. By Ehresmann theorem~\cite[Corollary 6.2.3]{huy}, it follows that $\pi\colon T_M \longrightarrow T_N$ is a locally trivial fibration. Since holomorphic local trivializations exist,~$\pi$~defines a holomorphic fiber bundle.

We will prove that the fibers of $\pi$, which are complex submanifolds of $T_M$, are complex tori, by constructing an explicit isomorphism between each fiber and a complex torus.
Fix $l\Gamma_N \in L/ \Gamma_N$. Then
\[\pi^{-1}(l\Gamma_N)=\{g\Gamma \in G/ \Gamma \mid p(g)\Gamma_N=l \Gamma_N\}.\]
Define a map
\[\psi_l\colon \ H \longrightarrow \pi^{-1}(l \Gamma_N), \psi_l(h)=(l,h)\Gamma.\]
Since $\Gamma_P$ is generated only by translations, it is a normal subgroup in $\Gamma$.
We obtain an induced map
\begin{gather*}
\bar{\psi_l}\colon\ H/\Gamma_P \simeq \mathbb T^{k} \longrightarrow \pi^{-1}(l \Gamma_N),\qquad
\bar{\psi_l}(h \Gamma_P)=(l,h)\Gamma.
\end{gather*}
It is clear that $\bar{\psi_l}$ is a biholomorphism.
Thus, the fiber of $\pi$ is a complex torus.

In conclusion, we can state the following.

\begin{Theorem}\label{ex_th}
Let $X$ be an Endo--Pajitnov manifold associated to a block diagonal matrix such that one of the blocks produces a $($smaller dimensional$)$ Endo--Pajitnov manifold $Y$. Then $X$ admits the structure of a holomorphic fiber bundle over $Y$. In particular, $X$ contains complex tori, as complex submanifolds.
\end{Theorem}

In the following, we provide a numerical example in the lowest possible dimension.

\begin{Example}\label{ex_curves}
We give an example\footnote{We are much grateful to Alexandru Gica for offering us this example.} of an Endo--Pajitnov manifold that contains complex curves.
Let $n=2$, $k=1$, and a diagonalizable matrix $M$
\[
M = \begin{pmatrix}
	N & 0 \\
	0 & P
\end{pmatrix},
\qquad
\text{where}\qquad
N= \begin{pmatrix}
\hphantom{-} 1 & 2 & -1\\
-1 & 0 & -2 \\
\hphantom{-} 0 & 1 & -1
\end{pmatrix}, \qquad
P=\begin{pmatrix}
0 & -1 \\
1 & \hphantom{-} 0
\end{pmatrix}.
\]
It is easy to see that $M \in \textrm{SL}(5, \mathbb{Z})$ and satisfies the special conditions from construction of Endo--Pajitnov manifold.
Thus, we obtain $T_M$, an Endo--Pajitnov manifold of complex dimension $3$.

On the other hand, $N\in \textrm{SL}(3, \mathbb{Z})$ has a single real eigenvalue, $\alpha$, and two complex conjugate eigenvalues $\beta_1$, $\overline{\beta}_1$, and hence it defines an Inoue surface of type $S^N$, call it $T_N=\mathbb{H} \times \mathbb{C}/G_N$.

As in the general case, we define the projection $\pi\colon T_M \longrightarrow T_N$
\[
\pi ([w,(z_1, z_2)])=[[w,z_1]], \qquad w \in \mathbb{H},\quad z_1, z_2 \in \mathbb{C}.
\]
Since $\pi$ is a holomorphic submersion, Theorem~\ref{ex_th} assures that $T_M$ projects over an Inoue surface, with complex curves as fibres.
\end{Example}

\section{Curves on Endo--Pajitnov manifolds}

In this section, we derive a necessary condition that the matrix $M$ must satisfy so that the manifold $T_M$ does not contain complex curves. The condition we find is algebraic, expressed in~terms of the components of the eigenvector $a$ associated with the real eigenvalue $\alpha$ of the matrix~$M$. The proof is similar to the one in~\cite{v_1}, where it was shown that no OT manifold can contain complex curves.

\begin{Theorem}\label{th_main}
Let $T_M$ be an Endo--Pajitnov manifold. If the components of the eigenvector~$a$~associated to the real eigenvalue $\alpha$ of the matrix $M$ are linearly independent over $\mathbb{Z}$, then there are no compact complex curves on $T_M$.
\end{Theorem}

\begin{proof}
The idea of the proof is the following. We construct an exact, semipositive (1,1)-form~$\omega$ on $T_M$ whose integral over any compact complex curve of $T_M$ will be necessarily nonnegative and such that any complex curve in $T_M$ should stay in a leaf of the null foliation of $\omega$. We then show that the stated condition implies that these leaves are isomorphic to $\mathbb{C}^n$, which contains no compact curves.

Here are the details.

{\it Step $1$. Construction of the form $\omega$.} We start by constructing a semipositive $(1,1)$-form $\tilde \omega$ on the universal cover $\tilde T_M:=\mathbb{H} \times \mathbb{C}^n$ of $T_M$, invariant by the action of the deck group $G_M$.

Let $(w, z_1, \dots, z_n)$ be the complex coordinates on $\tilde T_M$ and define $\varphi\colon \tilde T_M \longrightarrow \mathbb{R}$, by
\begin{equation*}
\varphi(w,z_1, \dots, z_n)=\frac{1}{\Im w}, \qquad w \in \mathbb{H}, \quad z_1, \dots, z_n \in \mathbb{C}.
\end{equation*}
It is clear that $\varphi(w,z_1, \dots, z_n)>0$ on $\tilde T_M $.

Define $\tilde{\omega}:=\mathrm{i} \6\Bar{\6} \log(\varphi)$.
In the above coordinates on $\tilde{T_M}$, $\tilde \omega$ is expressed as
\[\tilde{\omega}=\mathrm{i} \frac{1}{4(\Im w)^2} {\rm d}w \wedge {\rm d}\Bar{w}.\]
Note that using the ${\rm d}$ and ${\rm d}^c$ operators, we can rewrite
\[\tilde{\omega}=\frac{1}{2}{\rm d} {\rm d}^c \log \varphi,\]
and hence $\tilde{\omega}$ is an exact form on $\tilde T_M$.

Let us consider $\omega_\H$, the Poincar\'e metric on $\H$ and ${\rm pr}_1\colon \H \times \C^n \longrightarrow \H$ the projection onto the first factor. Then, we have
\begin{equation*}
 \tilde \omega={\rm pr}_1^*(\omega_\H).
\end{equation*}
Since ${\rm pr}_1$ is a holomorphic submersion, it follows that $\tilde \omega$ is semipositive definite.
Moreover, since $\omega_\H$ is invariant under translations and multiplications by real numbers, we obtain that $\tilde \omega$ is invariant under the action of $G_M$.
 Since $\tilde{\omega}$ is $G_M$-invariant, it is the pullback of an $(1,1)$-form~$\omega$ on~$T_M:= \tilde T_M/G_M$. Clearly, $\omega$~is an exact, semipositive $(1,1)$-form on~$T_M$.

{\it Step $2$. The action of the deck group on the leaves of the null foliation of $\tilde\omega$.}
 If $V=Z+A$, where $Z \in T\mathbb{H}$, $A \in T\mathbb{C}^n$ and $Z=X+\mathrm{i}Y$, then
\begin{equation} \label{nucleu}
\tilde{\omega}(V, JV) = \frac{\mathrm{i}}{4 (\Im w)^2} \cdot (-2\mathrm{i}) \, \mathrm{d}w(Z) \, \mathrm{d}\bar{w}(Z) = \frac{2}{4 (\Im w)^2} \big(|X|^2 + |Y|^2\big).
\end{equation}
 From \eqref{nucleu}, we obtain that any (maximal) leaf of the zero foliation of $\tilde{\omega}$ on $\tilde T_M$ is isomorphic to~$\mathbb{C}^n$.

Let $L=\{w\} \times \mathbb{C}^n$, for some fixed $w$, be such a leaf. We look at the image of the action of~$G_M$ on~$L$ and we determine its intersection with $L$.
By the description of $L$, for any $\sigma \in G_M$ such that $L \cap \sigma(L) \neq \varnothing$, the first coordinate of the points in $L$ coincide with the first coordinate of the points in $\sigma(L)$.
In general, $\sigma$ contains all generators of $G_M$. Its most general form is $\sigma=g_0^{s_0} \circ g_1^{s_1}\circ \dots \circ g_{2n+1}^{s_{2n+1}}$, where $s_i \in \mathbb{Z}$.

We show that $g_0$ cannot appear.
Indeed, the above mentioned coincidence of the first coordinates translates in the following equation:
\begin{equation*}
\alpha^{s_0}w+\sum_{i=1}^{2n+1} s_i a^i=w,
\end{equation*}
which is equivalent to
\begin{equation*}
(\alpha^{s_0}-1)w=-\sum_{i=1} ^{2n+1} s_i a^i.
\end{equation*}
Taking imaginary parts in the equation and using the fact that $w \in \H$, we necessarily obtain that $s_0=0$.
We conclude that $\sigma$ cannot contain the generator $g_0$. It is then obtained only from translations, $\sigma=g_1^{s_1}\circ \cdots \circ g_{2n+1}^{s_{2n+1}}$, and we have
\begin{equation*}
s_1 a^1+ \dots +s_{2n+1} a^{2n+1}=0,
\end{equation*}
a linear dependence relation over $\mathbb{Z}$ which contradicts the hypothesis.
Thus, we showed that $L \cap \sigma(L) = \varnothing$, for all $\sigma \in G_M$.

{\it Step $3$. The zero foliation of $\omega$ on $T_M$.}
Recall that $\omega$ is semipositive on $T_M$ (Step 1) and hence its integral on any compact complex curve $\gamma \subset T_M$ is nonnegative. By Stokes theorem, since $\omega$ is exact, this integral vanishes. Thus, $\omega$ vanishes on all closed complex curves in $T_M$. Equivalently, any compact complex curve in $T_M$ stays in a leaf of the zero foliation of $\omega$.

On the other hand, since $\tilde{\omega}$ is $G_M$-invariant, each leaf of the zero foliation of $\omega$ on $T_M$ is isomorphic to a component of the leaf of the zero foliation of $\Tilde{\omega}$ on $\tilde T_M$. Therefore, it is isomorphic with $\mathbb{C}^n$, which does not contain any compact complex submanifold.
\end{proof}

\begin{Remark}
	Clearly, Example \ref{ex_curves} does not satisfy the condition in Theorem~\ref{th_main}. At the moment, we cannot prove that the condition is also sufficient. However, we dare to propose the following:
\end{Remark}

\begin{Conjecture}
	Let $T_M$ be an Endo--Pajitnov manifold, with real eigenvalue $\alpha$. Then $T_M$ admits complex curves if and only if the components of the eigenvector $a$ are not linearly independent over $\mathbb{Z}$.
\end{Conjecture}

\begin{Remark}
 In~\cite[Example 5.8]{cos}, we constructed an example of a $4$-dimensional Endo--Pajitnov manifold admitting both pluriclosed and astheno-K\"ahler metrics. Moreover, the condition in Theorem~\ref{th_main} is satisfied, so the manifold contains no compact complex curves.
\end{Remark}

In a manner similar to~\cite{v_2}, we obtain a result concerning the existence of complex surfaces in Endo--Pajitnov manifolds.

\begin{Proposition} \label{sup}
 Let $T_M$ be an Endo--Pajitnov manifold without compact complex curves. Then~$T_M$ does not contain any closed complex surfaces except Inoue surfaces.
\end{Proposition}

\begin{proof}
 The idea of the proof is based on a result by Brunella about classification of surfaces of K\"ahler rank one.

 It was shown in~\cite{pajitnov1} that Endo--Pajitnov manifolds are non-K\"ahler. In the previous proof, we constructed a non-trivial closed semipositive $(1,1)$-form $\omega$ on $T_M$. The restriction of $\omega$ to any complex surface in $T_M$ yields a non-trivial closed semipositive $(1,1)$-form on that surface. Therefore, any compact complex surface in $T_M$ must have K\"ahler rank one.

 Compact surfaces of K\"ahler rank $1$ have been classified in~\cite{ct} and~\cite{br}. They can be
 \begin{enumerate}\itemsep=0pt
 \item Non-K\"ahlerian elliptic fibrations;
 \item certain Hopf surfaces, and their blow-ups;
 \item Inoue surfaces, and their blow-ups.
\end{enumerate}

By definition, elliptic fibrations contain curves, as do Hopf surfaces and their blow-ups. Since~$T_M$ does not contain any compact complex curves by hypothesis, it cannot contain any surface from the above classification except for Inoue surfaces.
\end{proof}

\subsection*{Acknowledgements}
My warm thanks to Liviu Ornea for suggesting this research topic and to Miron Stanciu for many useful discussions. I am grateful to the anonymous referees for their insightful and constructive suggestions, which have significantly improved the quality of this work.

\pdfbookmark[1]{References}{ref}
\LastPageEnding

\end{document}